\documentclass[11 pt,twoside, final]{amsart}
\copyrightinfo{0}{Iranian Mathematical Society}
\usepackage{amsmath,amsthm,amscd,amsfonts,amssymb,enumerate}
\usepackage{graphicx}		
\usepackage{color}
\usepackage[colorlinks]{hyperref}
 
\newtheorem{theorem}{Theorem}[section]
\newtheorem{proposition}[theorem]{Proposition}
\newtheorem{lemma}[theorem]{Lemma}
\newtheorem{corollary}[theorem]{Corollary}
\theoremstyle{definition}
\newtheorem{definition}[theorem]{Definition}
\newtheorem{example}[theorem]{Example}
\theoremstyle{remark}

\numberwithin{equation}{section}
 
 \commby{Hamid Pezeshk}
\begin{document}


\title[$\Delta$-weak character amenability]{$\Delta$-weak character amenability of certain Banach algebras}

\author[H. Sadeghi]{H. Sadeghi }
\address[Hamid Sadeghi]{Department of Mathematics, Faculty  of  Science,  University of Isfahan, Isfahan, Iran. }
\email{Sadeghi@sci.ui.ac.ir}

\date{
}

\begin{abstract}
In this paper we introduce the notion of $\Delta$-weak character amenable Banach algebras and  investigate $\Delta$-weak character amenability of certain Banach algebras such as projective tensor product $A\widehat{\otimes}B$, Lau product $A\times_TB$ for every Banach  algebra $ A$ and $B$, where $T$ be a homomorphism from $B$ into $A$,  abstract Segal algebras and module extension Banach algebras.$\vspace{.5cm}$\\
\textbf{2010 Mathematics subject classifucation:}  Primary: 46H25; Secondary: 46M10. 
\end{abstract} \maketitle
\section{\bf Introduction}

Let $A$ be a Banach algebra and let $\varphi\in \Delta(A)$, consisting of all nonzero character on $A$. The concept of $\varphi$-amenability was first introduced by Kaniuth \textit{et al}. in \cite{5}. Specifically, $A$ is called $\varphi$-amenable if there exist a $m\in A^{**}$ such that
\begin{enumerate}
\item $m(\varphi)=1$;
\item $m(f.a)=\varphi(a)m(f) \hspace{.2cm}(a\in A, f\in A^*).$
\end{enumerate}
  Monfared in \cite{502},  introduced and studied the notion of character amenable Banach algebra.  $A$ was called character amenable if it has a bounded right approximate identity and it is $\varphi$-amenable for all $\varphi\in \Delta(A)$.
Many aspects of $\varphi$-amenability have been investigated in \cite{4,5,6}. 
 
  Let $A$ be a Banach algebra and $\varphi\in \Delta(A)\cup \{0\}.$ Following \cite{3}, $A$ is called $\Delta$-weak $\varphi$-amenable if, there exists a $m\in A^{**}$ such that
  \begin{enumerate}
  \item $m(\varphi)=0 $;
  \item $m(\psi.a)=\psi(a)\hspace{.2cm}(a\in \ker(\varphi), \psi\in \Delta(A)).$
\end{enumerate} 
In this paper we use above definition with a slight difference. In fact we say that $A$ is called $\Delta$-weak $\varphi$-amenable if, there exists a $m\in A^{**}$ such that
  \begin{enumerate}
  \item $m(\varphi)=0 $;
  \item $m(\psi.a)=\psi(a)\hspace{.2cm}(a\in A, \psi\in \Delta(A)\setminus\{\varphi\}).$
\end{enumerate} 
 
The aim of the present work is to study $\Delta$-weak character amenability of certain Banach algebras such as projective tensor product $A\widehat{\otimes}B$, Lau product $A\times_TB$,  abstract Segal algebras and module extension Banach algebras.
Indeed, we show that $A\widehat{\otimes}B$ is $\Delta$-weak character amenable if and only if both $A$ and $B$ is $\Delta$-weak character amenable. It is also shown that if $A\times_TB$ is $\Delta$-weak character amenable, then so are $A$ and $B$. In the case that $T(B)$ is dense in $A$, we prove that $\Delta$-weak character amenability of $A$ and $B$ implies  $\Delta$-weak character amenability of $A\times_TB$.  For abstract Segal algebra $B$ with respect to $A$, we investigate relation  between $\Delta$-weak character amenability of $A$ and $B$. Finally, for a Banach algebra $A$ and $A$-bimodule $X$ we show that $A\oplus_1X$ is $\Delta$-weak character amenable if and only if $A$ is $\Delta$-weak character amenable.
\section{$\Delta$-weak character amenability of $A\widehat{\otimes}B$}
 We commence this section with the following definition:
\begin{definition}
 Let $A$ be a Banach algebra and $\varphi\in \Delta(A)\cup \{0\}.$ We say that $A$ is called $\Delta$-weak $\varphi$-amenable if, there exists a $m\in A^{**}$ such that
  \begin{enumerate}
  \item $m(\varphi)=0 $;
  \item $m(\psi.a)=\psi(a)\hspace{.2cm}(a\in A, \psi\in \Delta(A)\setminus\{\varphi\}).$
\end{enumerate} 
\end{definition}
\begin{example}
Let $A$ be a Banach algebra such that $\Delta(A)=\{\varphi,\psi\}$, where $\varphi\neq\psi$. Hence, by Theorem 3.3.14 of \cite{8}, there exists a $a_0\in A$ with $\varphi(a_0)=0$ and $\psi(a_0)=1$. Put $m=\widehat{a_0}$. Then $m(\varphi)=\widehat{a_0}(\varphi)=\varphi(a_0)=0$ and for every $a\in A$, we have 
$$m(\psi.a)=\widehat{a_0}(\psi.a)=\psi.a(a_0)=\psi(aa_0)=\psi(a).$$
Therefore $A$ is $\Delta$-weak $\varphi$-amenable.
\end {example}
\begin{example}
Let $A$ be a Banach algebra and $\varphi\in \Delta(A)\cup\{0\}$. Suppose that $A$ is a $\varphi$-amenable and has a bounded right approximate identity. By Corollary 2.3 of \cite{5}, $\ker(\varphi)$ has a bounded right approximate identity.  Let $(e_{\alpha})_{\alpha}$ be a bounded right approximate identity for $\ker(\varphi)$. If there exists $a_0\in A$ with $\varphi(a_0)=1$ and $\lim_{\alpha}\big|\psi(a_0e_{\alpha})-\psi(a_0)\big|=0$ for all $\psi\in \Delta(A)\setminus \{\varphi\}$, then $A$ is $\Delta$-weak $\varphi$-amenable. For see this let $m=w^*-\lim_{\alpha}(\widehat{e_{\alpha}})$. Now, we have 
$$m(\varphi)=\lim_{\alpha}\widehat{e_{\alpha}}(\varphi)=\lim_{\alpha}\varphi(e_{\alpha})=0,$$
and for every $\psi\in \Delta(A)\setminus \{\varphi\}$ and $a\in \ker(\varphi)$,
$$m(\psi.a)=\lim_{\alpha}\widehat{e_{\alpha}}(\psi.a)=\lim_{\alpha}\psi.a(e_{\alpha})=\lim_{\alpha}\psi(ae_{\alpha})=\psi(a).$$
Let $a\in A$. Then $a-\varphi(a)a_0\in \ker(\varphi)$ and for every $\psi\in \Delta(A)\setminus \{\varphi\}$, we have 
$$m\big(\psi.(a-\varphi(a)a_0)\big)=\psi\big(a-\varphi(a)a_0\big).$$
Therefore $m(\psi.a)=\psi(a)$. So $A$ is $\Delta$-weak $\varphi$-amenable.
\end{example}
 \begin{definition}
 Let $A$ be a Banach algebra. We say that $A$ is  $\Delta$-weak character amenable if it is $\Delta$-weak $\varphi$-amenable for every $\varphi\in \Delta(A)\cup \{0\}.$
 \end{definition}
 \begin{definition}
 Let $A$ be a Banach algebra. The net $(a_{\alpha})_{\alpha}$ in $A$ is called a $\Delta$-weak approximate identity if,  $|\varphi(aa_{\alpha})-\varphi(a)|\longrightarrow 0$, for each $a\in A$ and $\varphi\in \Delta(A).$
 \end{definition}
 
 Note that the approximate identity and $\Delta$-weak approximate identity of a Banach algebra are different. Jones and Lahr  proved that if $S=\Bbb Q^+$ the semigroup algebra $l^1(S)$ has a bounded $\Delta$-weak approximate identity, but it does not have any bounded or unbounded approximate identity (see \cite{9}).
 
The proof of the following theorem is omitted, since it can be proved in the same direction
of Theorem 2.2 of \cite{3}.
 \begin{theorem}\label{thm1}
 Let $A$ be a Banach algebra and $\varphi\in \Delta(A)\cup\{0\}$. Then $A$ is $\Delta$-weak $\varphi$-amenable if and only if $\ker(\varphi)$ has a bounded $\Delta$-weak approximate identity for $A$ $($i.e. there exists a net $(a_{\alpha})_{\alpha}\subseteq \ker(\varphi)$ such that $|\psi(aa_{\alpha})-\psi(a)|\longrightarrow 0$, for each $a\in A$ and $\psi\in \Delta(A)\setminus\{\varphi\})$.
 \end{theorem}
 For $f\in A^*$ and $g\in B^*$, let $f\otimes g$ denote the element of $(A\widehat{\otimes}B)^*$ satisfying $(f\otimes g)(a\otimes b)=f(a)g(b)$ for all $a\in A$ and $b\in B$. Note that, with this notion, 
 $$\Delta(A\widehat{\otimes}B)=\{\varphi\otimes \psi:\varphi\in \Delta(A), \psi\in \Delta(B)\}.$$
 
 \begin{theorem}
 Let $A$ and $B$  be Banach algebras and let $\varphi\in \Delta(A)\cup\{0\}$ and $ \psi\in \Delta(B)\cup\{0\}$. Then $A\widehat{\otimes}B$ is $\Delta$-weak $(\varphi\otimes \psi)$-amenable if and only if $A$ is $\Delta$-weak $\varphi$-amenable and $B$ is $\Delta$-weak $\psi$-amenable.
 \end{theorem}
 \begin{proof}
 Suppose that $A\widehat{\otimes}B$ is $\Delta$-weak $(\varphi\otimes \psi)$-amenable.  So, there exists $m\in (A\widehat{\otimes}B)^{**}$ such that 
 $$ m( \varphi\otimes \psi)=0, \hspace{.2cm}m(\varphi'\otimes \psi'.a\otimes b)=\varphi'\otimes \psi'(a\otimes b),$$
 for all $ a\otimes b\in A\widehat{\otimes}B, \varphi'\otimes \psi'\in \Delta(A\widehat{\otimes}B).$
   Choose $b_0\in B$ such that $\psi(b_0)=1$, and define $m_{\varphi}\in A^{**}$ by $m_{\varphi}(f)=m(f\otimes \psi)(f\in A^*)$. Then $m_{\varphi}(\varphi)=m(\varphi\otimes \psi)=0$ and for every $a\in A$ and $\varphi'\in \Delta(A)$, we have
   \begin{align*}
   m_{\varphi}(\varphi'.a)&=m(\varphi'.a\otimes \psi)=m((\varphi'.a\otimes \psi.b_0)\\
   &=m(\varphi'\otimes \psi.a\otimes b_0)=\varphi'\otimes \psi(a\otimes b_0)\\
   &=\varphi'(a).
\end{align*}    
Thus $A$ is $\Delta$-weak $\varphi$-amenable. By a similar argument as above we may show that $B$ is $\Delta$-weak $\psi$-amenable.

Conversely, let $A$ is $\Delta$-weak $\varphi$-amenable and $B$ is $\Delta$-weak $\psi$-amenable. By Theorem \ref{thm1}, there are bounded nets $(a_{\alpha})_{\alpha}$ and $(b_{\beta})_{\beta}$ in $\ker(\varphi)$ and $\ker(\psi)$, respectively, such that $\big|\varphi'(aa_{\alpha})-\varphi'(a)\big|\longrightarrow 0$ and $\big|\psi'(bb_{\beta})- \psi'(b)\big|\longrightarrow 0$ for all $a\in A, b\in B, \varphi'\in \Delta(A), \psi'\in \Delta(B)$, where $\varphi'\neq \varphi$  and $\psi'\neq \psi$. Consider the bounded net $((a_{\alpha},b_{\beta}))_{(\alpha,\beta)}$ in $A\widehat{\otimes}B$.  We show that $((a_{\alpha},b_{\beta}))_{(\alpha,\beta)}$ is a bounded $\Delta$-weak approximate identity for $A\widehat{\otimes}B$ in $\ker(\varphi\otimes \psi)$. Indeed, let $\|a_{\alpha}\|\leq M_1, \|b_{\beta}\|\leq M_2$ and let $F=\sum_{i=1}^Nc_i\otimes d_i\in A\widehat{\otimes}B$. For every $\varphi'\in \Delta(A)\setminus\{0\}$ and $\psi'\in \Delta(B)\setminus \{0\}$, we have
\begin{align*}
&|\varphi'\otimes \psi'(F.a_{\alpha}\otimes b_{\beta})-\varphi'\otimes \psi'(F)|\\
&\quad=\big|\sum_{i=1}^N\Big[\big(\varphi'(c_ia_{\alpha})-\varphi'(c_i)\big)\psi'(d_ib_{\beta})+\varphi'(c_i)\big(\psi'(d_ib_{\beta})-\psi'(d_i)\big)\Big]\big|\\
&\quad\leq \sum_{i=1}^NM_2\|d_i\|\|\psi'\|\big|\varphi'(c_ia_{\alpha})-\varphi'(c_i)\big|+\sum_{i=1}^N\|\varphi'\|\|c_i\|\big|\psi'(d_ib_{\beta})-\psi'(d_i)\big|\\
&\quad\longrightarrow 0.
\end{align*} 
Now let $G\in A\widehat{\otimes}B$, so there exist sequences $(c_i)_i\subseteq A$ and $(d_i)_i\subseteq B$ such that $G=\sum_{i=1}^{\infty}c_i\otimes d_i$ with $\sum_{i=1}^{\infty}\|c_i\|\|d_i\|<\infty.$ Let $\varepsilon>0$ be given, we choose $N\in \Bbb N$ such that $\sum_{i=N+1}^{\infty}\|c_i\|\|d_i\|<\varepsilon/4M_1M_2\|\varphi'\|\|\psi'\|.$ Put $F=\sum_{i=1}^Nc_i\otimes d_i$. Since $|\varphi'\otimes \psi'(F.a_{\alpha}\otimes b_{\beta})-\varphi'\otimes \psi'(F)|\longrightarrow 0$, it follows that there exists $(\alpha_0,\beta_0)$ such that $|\varphi'\otimes \psi'(F.a_{\alpha}\otimes b_{\beta})-\varphi'\otimes \psi'(F)|<\varepsilon/2$ for all $(\alpha,\beta)\geq (\alpha_0,\beta_0).$ Now for such a $(\alpha,\beta)$, we have
\begin{align*}
&\big|\varphi'\otimes \psi'(G.a_{\alpha}\otimes b_{\beta})-\varphi'\otimes \psi'(G)\big|\\
&\quad=\big|\varphi'\otimes \psi'(F.a_{\alpha}\otimes b_{\beta})-\varphi'\otimes \psi'(F)\\
&\quad\quad+\sum_{i=1+N}^{\infty}\big(\varphi'(c_ia_{\alpha})\psi'(d_ib_{\beta})-\varphi'(c_i)\psi'(d_i)\big)\big|\\
&\quad\leq \varepsilon/2+2M_1M_2\|\varphi'\|\|\psi'\|\sum_{i+N}^{\infty}\|c_i\|\|d_i\|\leq \varepsilon/2+\varepsilon/2=\varepsilon.
\end{align*}
Hence $\big|\varphi'\otimes \psi'(G.a_{\alpha}\otimes b_{\beta})-\varphi'\otimes \psi'(G)\big|\longrightarrow 0$. Also, clearly $\big|\varphi'\otimes \psi'(G.a_{\alpha}\otimes b_{\beta})-\varphi'\otimes \psi'(G)\big|\longrightarrow 0$ for $\varphi'=0$ and $\psi'=0$ and it is easy to see that $((a_{\alpha},b_{\beta}))_{(\alpha,\beta)}\subset \ker(\varphi\otimes \psi).$ So $((a_{\alpha},b_{\beta}))_{(\alpha,\beta)}$ is a bounded $\Delta$-weak approximate identity for $A\widehat{\otimes}B$ in $\ker(\varphi\otimes \psi)$. Therefore  $A\widehat{\otimes}B$ is $\Delta$-weak $(\varphi\otimes \psi)$-amenable, again by Theorem \ref{thm1}.
 \end{proof}
 \begin{corollary}
 Let $A$ and $B$  be Banach algebras. Then $A\widehat{\otimes}B$ is $\Delta$-weak character amenable if and only if both $A$ and $B$ are $\Delta$-weak  character amenable.
 \end{corollary}
 
 \section{$\Delta$-weak character amenability of $A\times_TB$}
 
The Lau product of two Banach algebras was first introduced by Lau \cite{51}. 
Extension to arbitrary Banach algebras was proposed by M. S. Monfared \cite{6} with the notation $A\times_{\theta}B$, where $\theta$ is a non-zero multiplication linear functional on $B$. 

Recently, in the case where $A$ is commutative, a new extension of Lau product introduced and has been studied by Bhatt and Dabhi \cite{53}. In \cite{54}, the authors, by slight change in the definition of multiplication given by Bhatt and Dabhi on $A\times_{\varphi}B$, investigated the two notions, biprojectivity and biflatness on $ A\times_{\varphi}B$ for an arbitrary Banach algebra $A$, where $\varphi$ is a Banach algebra homomorphism from $B$ into $A$.

Let $A$ and $B$ be two Banach algebras and let $T\in$ Hom$(B,A)$, the space consisting of all Banach algebra homomorphism from $B$ into $A$. Moreover, assume that $\|T\|\leq 1$. Following \cite{54}, the cartesian product space $A\times B$ equipped with the following algebra multiplication
$$(a_1,b_1).(a_2,b_2)=(a_1a_2+a_1T(b_2)+T(b_1)a_2,b_1b_2), \hspace{.2cm}(a_1,a_2\in A, b_1,b_2\in B),$$
and the norm $\|(a,b)\|=\|a\|_A+\|b\|_B$, defined a Banach algebra which is denoted by $A\times_TB$.

We note that the dual space $(A\times_TB)^*$ can be identified with $A^*\times B^*$, via 
$$\langle(f,g),(a,b)\rangle=\langle a,f\rangle+\langle b,g\rangle\hspace{.2cm}(a\in A, f\in A^*, b\in B, g\in B^*).$$  
 For more details see Theorem 1.10.13 of \cite{72}.
 Moreover, $(A\times_TB)^*$ is a $(A\times_TB)$-bimodule with the module operations given by 
 \begin{eqnarray}\label{1}
(f,g).(a,b)=\big(f.a+f.T(b),f\circ (L_aT)+g.b\big)\hspace{.2cm}, 
\end{eqnarray}
\begin{eqnarray}\label{2}
(a,b).(f,g)=\big(a.f+T(b).f,f\circ (R_aT)+b.g\big)\hspace{.2cm},
\end{eqnarray}
for all $a\in A, b\in B$ and  $f\in A^*,g\in B^*$ where $L_aT:B\rightarrow A$ and $R_aT:B\rightarrow A$ are defined by $L_aT(b)=aT(b)$ and $R_aT(b)=T(b)a \hspace{.2cm}( b\in B)$. 

 \begin{proposition}
 Let $A$ and $B$ be Banach algebras and $T\in \mathtt{Hom}(B,A)$. Then $A\times_TB$ has  a $\Delta$-weak approximate identity if and only if $A$ and $B$ have $\Delta$-weak approximate identities.
 \end{proposition}
 \begin{proof}
 Let $((a_{\alpha},b_{\alpha}))_{\alpha}$ be a $\Delta$-weak approximate identity for $A\times_TB$. For every $\psi\in \Delta(B)$ and $b\in B$ we have,
 $$\big|\psi(bb_{\alpha})-\psi(b)\big|=\big|(0,\psi)\big((0,b)(a_{\alpha},b_{\alpha})\big)-(0,\psi)(0,b)\big|\longrightarrow 0.$$
 Then $(b_{\alpha})_{\alpha}$ is a $\Delta$-weak approximate identity for $B$. Also for every $\varphi\in \Delta(A)$ and $a\in A$,
 \begin{align*}
 \big|\varphi\big(a(a_{\alpha}+T(b_{\alpha}))\big)-\varphi(a)\big|&=\big|\varphi\big(aa_{\alpha}+aT(b_{\alpha})\big)-\varphi(a)\big|\\
 &=\big|(\varphi,\varphi\circ T)\big((a,0)(a_{\alpha},b_{\alpha})\big)-(\varphi,\varphi\circ T)(a,0)|\\
 &\longrightarrow 0.
\end{align*}  
 Thus $\big|\varphi\big(a(a_{\alpha}+T(b_{\alpha}))\big)-\varphi(a)\big|\longrightarrow 0.$ Therefore $(a_{\alpha}+T(b_{\alpha}))_{\alpha}$ is a $\Delta$-weak approximate identity for $A$.
 
 Conversely, suppose that $(a_{\alpha})_{\alpha}$ and $(b_{\beta})_{\beta}$ are $\Delta$-weak approximate identity for $A$ and $B$, respectively. We claim that $((a_{\alpha}-T(b_{\beta}),b_{\beta}))_{(\alpha,\beta)}$ is a $\Delta$-weak approximate identity for $A\times_TB$. In fact for every $a\in A, b\in B$ and $\varphi\in \Delta(A)$, we have
 \begin{align*}
 \big|(\varphi,\varphi\circ T)&\big((a,b)(a_{\alpha}-T(b_{\beta}),b_{\beta})\big)-(\varphi,\varphi\circ T)(a,b)\big|\\
 &=\big|(\varphi,\varphi\circ T)\big(aa_{\alpha}+T(b)a_{\alpha}-T(bb_{\beta}),b_{\beta}b)\big)-(\varphi,\varphi\circ T)(a,b)\big|\\
 &=\big|\varphi(aa_{\alpha})+\varphi(T(b)a_{\alpha})-\varphi(a)-\varphi\circ T(b)\big|\\
 &\leq \big|\varphi(aa_{\alpha})-\varphi(a)\big|+\big|\varphi(T(b)a_{\alpha})-\varphi( T(b))\big|\longrightarrow 0.
 \end{align*}
 Similarly, we may show that $\big|(0,\psi)\big((a,b)(a_{\alpha}-T(b_{\beta}),b_{\beta})\big)-(0,\psi)(a,b)\big|\longrightarrow 0$, for all $\psi\in \Delta(B)$. Therefore $((a_{\alpha}-T(b_{\beta}),b_{\beta}))_{(\alpha,\beta)}$ is a $\Delta$-weak approximate identity for $A\times_TB$.
 \end{proof} 
 
 \begin{theorem}
 Let $A$ and $B$ be two Banach algebras and $T\in \mathtt{Hom}(B,A)$. If $A\times_TB$ is $\Delta$-weak character amenable, then  both $A$ and $B$ are $\Delta$-weak character amenable. In the case that $T(B)$ is dense in $A$ the converse is also valid.
 \end{theorem}
 \begin{proof}
 Suppose that $A\times_TB$ is $\Delta$-weak character amenable. Let $\varphi\in \Delta(A)\cup \{0\}$. Then there exists $m\in (A\times_TB)^{**}$ such that $m(\varphi,\varphi\circ T)=0$ and $m(h.(a,b))=h(a,b)$ for all $(a,b)\in A\times_TB$ and $h\in \Delta(A\times_TB)$, where $h\neq (\varphi,\varphi\circ T)$. Define $m_{\varphi}\in A^{**}$ by $m_{\varphi}(f)=m(f,f\circ T) (f\in A^*)$. For every $a\in A$ and  $\varphi'\in \Delta(A)$, we have 
 \begin{align*}
 m_{\varphi}(\varphi'.a)&=m\big(\varphi'.a,(\varphi'.a)\circ T\big)\\
 &=m\big(\varphi'.a,\varphi'\circ L_{a}T\big)\\
 &=m\big((\varphi',\varphi'\circ T).(a,0)\big)\\
 &=(\varphi',\varphi'\circ T)(a,0)\\
 &=\varphi'(a).
 \end{align*}
 Also $m_{\varphi}(\varphi)=m(\varphi,\varphi\circ T)=0$. Thus $A$ is a $\Delta$-weak $\varphi$-amenable. Therefore $A$ is $\Delta$-weak character amenable.
 
 Let $\psi\in \Delta(B)\cup\{0\}$. From the $\Delta$-weak character amenability of $A\times_TB$ it follows that there exists a $m\in (A\times_TB)^{**}$ such that $m(0,\psi)=0$ and $m(h.(a,b))=h(a,b)$ for all $(a,b)\in A\times_TB$ and $h\in \Delta(A\times_TB)$, where $h\neq (0,\psi)$.  Define $m_{\psi}\in B^{**}$ by $m_{\psi}(g)=m(0,g)$. So $m_{\psi}(\psi)=m(0,\psi)=0$ and 
 $$m_{\psi}(\psi'.b)=m(0,\psi'.b)=m\big((0,\psi').(0,b)\big)=(0,\psi')(0,b')=\psi'(b),$$
 for all $b\in B$ and $\psi'\in \Delta(B)$. Therefore  $B$ is $\Delta$-weak character amenable.
 
 Conversely, let $A$ and $B$ be $\Delta$-weak character amenable. We show that for every $h\in \Delta(A\times_TB)$, $A\times_TB$ is $\Delta$-weak $h$-amenable. For see this we first assume that $h=(0,\psi)$, where $\psi\in \Delta(B)$. Since $B$ is $\Delta$-weak character amenable by Theorem \ref{thm1} there exists a net $(b_{\beta})_{\beta}\subseteq \ker \psi$ such that $|\psi'(bb_{\beta})-\psi'(b)|\longrightarrow 0$, for all $b\in B$ and $\psi'\in \Delta(B)$, where $\psi'\neq \psi$. We claim that $((0,b_{\beta}))_{\beta}\subseteq \ker h$ is a bounded $\Delta$-weak approximate identity for $A\times_TB$. Indeed, for every $a\in A, b\in B$ and $\psi'\in \Delta(B)$, we have 
 $$\big|(0,\psi')\big((a,b)(0,b_{\beta})\big)-(0,\psi')(a,b)\big|=\big|\psi'(bb_{\beta})-\psi'(b)\big|\longrightarrow 0.$$
 Also from density of $T(B)$ in $A$, it follows that for every $a\in A$, there exist a net $(b_i)_i$ in $B$ such that $\lim_iT(b_i)=a$. Hence for every $a\in A, b\in B$ and  $\varphi\in \Delta(A)$, we have 
 \begin{align*}
 \big|(\varphi,\varphi\circ T)\big(&(a,b)(0,b_{\beta})\big)-(\varphi,\varphi\circ T)(a,b)\big|\\
 &=\big|(\varphi,\varphi\circ T)\big(aT(b_{\beta}),bb_{\beta}\big)-(\varphi,\varphi\circ T)(a,b)\big|\\
 &=\big|\varphi\big(aT(b_{\beta})\big)+\varphi\circ T(bb_{\beta})-\varphi(a)-\varphi\circ T(b)\big|\\
 &\leq \big|\varphi\big(aT(b_{\beta})\big)-\varphi(a)\big|+\big|\varphi\circ T(bb_{\beta})-\varphi\circ T(b)\big|\\
 &=\lim_i\big|\varphi\circ T(b_ib_{\beta})-\varphi\circ T(b_i)\big|+\big|\varphi\circ T(bb_{\beta})-\varphi\circ T(b)\big|\longrightarrow 0.
 \end{align*}
 So, Theorem \ref{thm1} implies that $A\times_TB$ is $\Delta$-weak $(0,\psi)$-amenable. 
 
 Now let $h=(\varphi,\varphi\circ T)$, where $\varphi\in \Delta(A)$. Since $A$ is $\Delta$-weak $\varphi$-amenable by Theorem \ref{thm1}, there exists a net $(a_{\alpha})_{\alpha}\subseteq \ker \varphi$ such that $|\varphi'(aa_{\alpha})-\varphi'(a)|\longrightarrow 0$, for all $a\in A$ and $\varphi'\in \Delta(A)$, where $\varphi'\neq \varphi$. Also since $B$ is $\Delta$-weak $\varphi\circ T$-amenable again by Theorem \ref{thm1}, there exists a net $(b_{\beta})_{\beta}\subseteq \ker (\varphi\circ T)$ such that $|\psi'(bb_{\beta})-\psi'(b)|\longrightarrow 0$, for all $b\in B$ and $\psi'\in \Delta(B)$, where $\psi'\neq \varphi\circ T$. We show that $((a_{\alpha}-T(b_{\beta}),b_{\beta}))_{(\alpha,\beta)}$ is a bounded $\Delta$-weak approximate identity for $A\times_TB$ in $\ker(\varphi,\varphi\circ T)$. It is easy to see that $((a_{\alpha}-T(b_{\beta}),b_{\beta}))_{(\alpha,\beta)}\subseteq \ker(\varphi,\varphi\circ T)$  and is bounded. For every $a\in A, b\in B$ and $\psi'\in \Delta(B)$, we have 
 $$\big|(0,\psi')\big((a,b)(a_{\alpha}-T(b_{\beta}),b_{\beta})\big)-(0,\psi')\big(a,b\big)\big|=\big|\psi'(bb_{\beta})-\psi'(b)\big|\longrightarrow 0,$$
 and for every $\varphi'\in \Delta(A)$,
 \begin{align*}
 \big|(&\varphi',\varphi'\circ T)\big((a,b)(a_{\alpha}-T(b_{\beta}),b_{\beta})\big)-(\varphi',\varphi'\circ T)\big((a,b)\big)\big|\\
 &= \big|(\varphi',\varphi'\circ T)\big(aa_{\alpha}+T(b)a_{\alpha}-T(bb_{\beta}),bb_{\beta}\big)-(\varphi',\varphi'\circ T)\big((a,b)\big)\big|\\
 &=\big|\varphi'(aa_{\alpha})+\varphi'(T(b)a_{\alpha})-\varphi'\circ T(bb_{\beta})+\varphi'\circ T(bb_{\beta})-\varphi'(a)-\varphi'\circ T(b)\big|\\
 &\leq \big|\varphi'(aa_{\alpha})-\varphi'(a)\big|+\big|\varphi'(T(b)a_{\alpha})-\varphi'(T(b))\big|\longrightarrow 0.
 \end{align*}
 Then $((a_{\alpha}-T(b_{\beta}),b_{\beta}))_{(\alpha,\beta)}$ is a bounded $\Delta$-weak approximate identity for $A\times_TB$ in $\ker(\varphi,\varphi\circ T)$. Again Theorem \ref{thm1}, yield that $A\times_TB$ is $\Delta$-weak $(\varphi,\varphi\circ T)$-amenable. Therefore $A\times_TB$ is $\Delta$-weak character amenable.
 \end{proof}
 \section{$\Delta$-weak character amenability of abstract Segal algebras}
 We start this section with the basic definition of abstract Segal algebra; see \cite{2} for more details.
 Let $(A,\|.\|_A)$ be a Banach algebra. A Banach algebra $(B,\|.\|_B)$ is an abstract Segal algebra with respect to $A$ if:
 \begin{enumerate}
 \item $B$ is a dense left ideal in $A$;
 \item there exists $M>0$ such that $\|b\|_A\leq M\|b\|_B$ for all $b\in B$;
 \item there exists $C>0$ such that $\|ab\|_B\leq C\|a\|_A\|b\|_B$ for all $a,b\in B$.
\end{enumerate} 

Several authors have studied various notions of amenability for abstract Segal algebras; see, for example, \cite{7,1}.
 
 To prove our next result we need to quote the following lemma from \cite{1}.
 \begin{lemma}
 Let $A$ be a Banach algebra and let $B$ be an abstract Segal algebra with respect to $A$. Then $\Delta(B)=\{\varphi|_B:\varphi\in \Delta(A)\}$.
 \end{lemma}
 \begin{theorem}
 Let $A$ be a Banach algebra and let $B$ be an abstract Segal algebra with respect to $A$. If $B$ is $\Delta$-weak character amenable, then so is $A$. In the case that $B^2$ is dense in $B$ and $B$ has a bounded approximate identity the converse is also valid.
 \end{theorem}
 \begin{proof}
  Let $\varphi \in\Delta(A)$. Since $B$ is $\Delta$-weak $\varphi|_B$-amenable by Theorem \ref{thm1}, there exists a $\Delta$-weak approximate identity  $(b_{\alpha})_{\alpha}$ in $\ker(\varphi|_B)$ such that $\Big|\psi|_B(bb_{\alpha})-\psi|_B(b)\Big|\longrightarrow 0$, for all $b\in B$ and $\psi\in \Delta(A)$, where $\psi\neq \varphi|_B$. Let $\psi\in \Delta(A)$ and $a\in A$.  From the density of $B$ in $A$ it follows that there exists a net $(b_i)_i\subseteq B$ such that $\lim_ib_i=a$. So
 $$\Big|\psi(ab_{\alpha})-\psi(a)\Big|=\lim_i\Big|\psi|_B(b_ib_{\alpha})-\psi|_B(b_i)\Big|\longrightarrow 0.$$
  Thus $(b_{\alpha})_{\alpha}$ is a $\Delta$-weak approximate identity for $A$. Therefore $A$ is $\Delta$-weak $\varphi$-amenable. Then $A$ is $\Delta$-weak character amenable.
 
 Conversely, suppose that $A$ is $\Delta$-weak character amenable. Let $\varphi|_B\in \Delta(B)$. By Theorem\ref{thm1}, there exists a $\Delta$-weak approximate identity  $(a_{\alpha})_{\alpha}$ in $\ker(\varphi)$ such that $\Big|\psi(aa_{\alpha})-\psi(a)\Big|\longrightarrow 0$, for all $a\in A$ and $\psi\in \Delta(A)$, where $\psi\neq \varphi$.  Let $(e_i)_i$ be a bounded approximate identity for $B$ with bound $M> 0.$ Set $b_{\alpha}=\lim_i(e_ia_{\alpha}e_i),$ for all $\alpha$. We claim that $(b_{\alpha})_{\alpha}$ is a $\Delta$-weak approximate identity for $B$ in $\ker(\varphi|_B)$. From the facts that $B^2$ is dense in $B$ and continuity of $\varphi$ it follows that $b_{\alpha}\subseteq \ker(\varphi|_B)$. Moreover, for every $\psi|_B\in \Delta(B)$ and $b\in B$, we have 
 \begin{align*}
 \Big|\psi|_B(bb_{\alpha})-\psi|_B(b)\Big|&=\lim_i\Big|\psi|_B(be_ia_{\alpha}e_i)-\psi|_B(b)\Big|\\
 &=\lim_i\Big|\psi|_B(be_i^2a_{\alpha})-\psi|_B(b)\Big|\\
 &=\Big|\psi|_B(ba_{\alpha})-\psi|_B(b)\Big|\longrightarrow 0.
 \end{align*}
Hence, $(b_{\alpha})_{\alpha}$ is a $\Delta$-weak approximate identity for $B$.  So $B$ is $\Delta$-weak $\varphi|_B$-amenable. Therefore $B$ is $\Delta$-weak  character amenable.
 \end{proof}

 \section{$\Delta$-weak character amenability of module extension Banach algebras}
 
 Let $A$ be a Banach algebra and $X$ be a Banach $A$-bimodule. The $l^1$-direct sum of $A$ and $X$, denoted by $A\oplus_1X$, with the product defined by
 $$(a,x)(a',x')=(aa',a.x'+x.a')\hspace{1cm}(a,a'\in A, x,x'\in X),$$
 is a Banach algebra that is called the module extension Banach algebra of $A$ and $X$.
 
 Using the fact that the element $(0,x)$ is nilpotent in $A\oplus_1X$ for all $x\in X$, it is easy to verify that $$\Delta(A\oplus_1X)=\{\tilde{\varphi}:\varphi\in \Delta(A)\},$$
  where $\tilde{\varphi}(a,x)=\varphi(a)$ for all $a\in A$ and $x\in X$.
\begin{theorem}
Let $A$ be a Banach algebra and $X$ be a Banach $A$-bimodule. Then $A\oplus_1X$ is $\Delta$-weak character amenable if and only if $A$ is $\Delta$-weak character amenable.
\end{theorem} 
  \begin{proof}
  Suppose that $A$ is $\Delta$-weak character amenable. Let $\tilde{\varphi}\in \Delta(A\oplus_1X)$. By Theorem \ref{thm1}, there exists a $\Delta$-weak approximate identity  $(a_{\alpha})_{\alpha}$ for $A$ in $\ker(\varphi)$ such that $\big|\psi(aa_{\alpha})-\psi(a)\big|\longrightarrow 0$, for all $a\in A$ and $\psi\in \Delta(A)$, where $\psi\neq \varphi$. We show that $(a_{\alpha},0)_{\alpha}$ is a $\Delta$-weak approximate identity for $A\oplus_1X$ in $\ker(\tilde{\varphi})$. Clearly, $(a_{\alpha},0)_{\alpha}\subseteq \ker(\tilde{\varphi}).$ For every $a\in A, x\in X$ and $\tilde{\psi}\in \Delta(A\oplus_1X)$, we have
  \begin{align*}
  \big|\tilde{\psi}\big((a,x)(a_{\alpha},0)\big)-\tilde{\psi}(a,x)\big|&=\big|\tilde{\psi}\big(aa_{\alpha},x.a_{\alpha}\big)-\tilde{\psi}(a,x)\big|\\
  &=\big|\psi(aa_{\alpha})-\psi(a)\big|\longrightarrow 0.
  \end{align*}
  Thus $(a_{\alpha},0)_{\alpha}$ is a $\Delta$-weak approximate identity. Hence $A\oplus_1X$ is $\Delta$-weak $\tilde{\varphi}$-amenable. Therefore $A\oplus_1X$ is $\Delta$-weak character amenable. 
  
  For the converse, let $\varphi\in \Delta(A)$. Then by Theorem \ref{thm1} there exists a $\Delta$-weak approximate identity  $(a_{\alpha},x_{\alpha})_{\alpha}$ in $\ker(\tilde{\varphi})$ such that $\big|\tilde{\psi}\big((a,x)(a_{\alpha},x_{\alpha})\big)-\tilde{\psi}(a,x)\big|\longrightarrow 0$, for all $a\in A, x\in X$ and $\tilde{\psi}\in \Delta(A\oplus_1X)$, where $\tilde{\psi}\neq \tilde{\varphi}$. So, 
  \begin{align*}
  \big|\psi(aa_{\alpha})-\varphi(a)\big|&=\big|\tilde{\psi}(aa_{\alpha},a.x_{\alpha}+x.a_{\alpha})-\tilde{\psi}(a,x)\big|\\
  &=\big|\tilde{\psi}\big((a,x)(a_{\alpha},x_{\alpha})\big)-\tilde{\psi}(a,x)\big|\longrightarrow 0,
  \end{align*}
  for all $a\in A$ and $\psi\in \Delta(A)$. Moreover, $\varphi(a_{\alpha})=\tilde{\varphi}(a_{\alpha},x_{\alpha})=0$, for all $\alpha$. Thus $(a_{\alpha})_{\alpha}\subseteq\ker(\varphi)$ is $\Delta$-weak approximate identity for $A$. Again by Theorem \ref{thm1}, we conclude that $A$ is $\Delta$-weak $\varphi$-amenable. Therefore  $A$ is $\Delta$-weak character amenable.
  \end{proof}
{\bf Acknowledgement.}  We are grateful to the office of Graduate studies of the university of Isfahan for their support. 

\bibliographystyle{amsplain}

\end{document}